\documentclass[11pt]{amsart}

\usepackage{pstricks, graphicx}
\usepackage{amsmath, amsfonts, amssymb, stmaryrd}

\newtheorem{theorem}{Theorem}
\newtheorem{lemma}[theorem]{Lemma}
\newtheorem{proposition}[theorem]{Proposition}

\theoremstyle{definition}

\theoremstyle{remark}

\numberwithin{equation}{section}

\newcommand{\Prob}{\mathrm{Prob}}
\newcommand{\E}{\mathrm{E}}
\newcommand{\Var}{\mathrm{Var}}

\begin{document}

\title[Approximating \scalebox{1.5}{$\pi$} with stocks]{Approximately 
  \scalebox{1.5}{$\pi$} proofs that the \\ 
  stock market can approximate \scalebox{1.5}{$\pi$}} 

\author{Sami Assaf}
\address{University of Southern California, 
  Department of Mathematics, 
  3620 S. Vermont Ave. KAP 104, 
  Los Angeles, CA 90089-2532}
\email{shassaf@usc.edu}





\begin{abstract}
  We give three derivations of P{\'o}lya's approximation for the
  expected range of a simple random walk in one dimension. This result
  allows for an estimation of the volatility of a financial instrument
  from the difference between the high and low prices, or,
  equivalently, for an estimation of $\pi$ from the ratio of the
  volatility to the difference between high and low prices.
\end{abstract}

\maketitle

\section{Computing volatility}
\label{s:SP}

The \textit{volatility} of a financial instrument, for example the
S\&P 500 stock index or one of its constituent stocks, is a measure of
the degree by which the price of the instrument fluctuates over a
given period of time. Mathematically, volatility is the variance of
the price regarded as a random walk.

Volatility is the key component in options pricing, but it is also
vital for determining the underlying risk of a position and for
determining optimal asset allocation for a portfolio. Therefore having
accurate volatility measurements and forecasts is crucial for the
financial sector.

In practice, direct application of the mathematical definition of
variance to compute historical volatility is complicated by the sheer
volume of trades and inaccurate or missing data. This leads
practicioners to compute historical volatility based on a small
quantization parameter. For example, one can consider the price every
one second and compute variance based on that. Of course, there is an
obvious trade off, where smaller parameters offer more accurate
estimates but require more intensive computations.

A relatively simple, yet surprisingly effective, method for
forecasting future volatility is to take a moving average of
historical volatility. For example, a \textit{simple moving average}
of historical volatility is given by
\begin{equation}
  \mathrm{SMA}_n(V_t) = \frac{1}{n} \sum_{t=1}^{n} V_t ,
  \label{e:vol}
\end{equation}
where $V_t$ denotes historical volatility and $n$ is the length of the
window.

Now instead consider the \textit{range series} given by the difference
between the daily high and low prices for an instrument. This data is
freely available at the close of each trading day. Then one can
compute a simple moving average for the daily range, denoted $R_t$, by
\begin{equation}
  \mathrm{SMA}_n(R_t) = \frac{1}{n} \sum_{t=1}^{n} R_t .
  \label{e:range}
\end{equation}
On any reasonably liquid instrument, that is, something with a high
volume of daily trades, one notices that the ratio of these predictors
over a trading month (on average 21 days) is constant. The precise
approximation is
\begin{equation}
  \frac{\mathrm{SMA}_{21}(V_t)}{\mathrm{SMA}_{21}(R_t)^2} \approx \frac{\pi}{8}.
  \label{e:pi}
\end{equation}
Therefore, if an irrational trader were so inclined, he could use the
volatility and the high and low prices of an instrument to estimate
$\pi$. In practice, however, a rational trader is likely more
interested in efficiently and accurately estimating future
volatility.

Since calculating historical daily volatility is computationally
intensive and is dependent upon a timely and accurate (and expensive)
data feed, independent traders without access to such a feed, or who
are sensitive to the cost of such a feed, as well as traders at larger
firms looking to increase efficiency without losing accuracy, can make
use of \eqref{e:pi}, or, rather, the following mathematical
explanation of it.

\begin{theorem}
  Letting $\sigma^2$ denote the variance of a random walk and $\Delta$
  denote the range of values, we have 
  \begin{equation}
    \E(\Delta) \sim  \sigma\sqrt{\frac{8}{\pi}}.
    \label{e:main}
  \end{equation}
  \label{t:main}
\end{theorem}

The prudent trader can use the range series $R_t$, computed using only
$2$ data points, as a surrogate for the volatility series $V_t$,
computed using at least $21600$ data points (for a $6$ hour trading
day with $1$ second quantization).

\section{Approximating volatility}
\label{s:surrogate}

We turn our focus now to determining the accuracy of estimating the
variance of a random walk using the range of values attained. To
begin, we consider a simple, symmetric, one-dimensional random walk on
the integers.

Let $X$ be a discrete, symmetric, one-dimensional random variable. For
example, let $X$ take values $\{\pm 1\}$ each with probability $1/2$,
i.e.
\begin{equation}
  \Prob\{X = +1\} = \frac{1}{2} \hspace{3em} 
  \Prob\{X = -1\} = \frac{1}{2}.
  \label{e:prob}
\end{equation}
The \emph{expectation} of $X$ is $\E(X) = \sum_x x \Prob\{X =
x\}$. For the example,
\begin{equation}
  \E(X) = (+1) \cdot \frac{1}{2} + (-1) \cdot \frac{1}{2} = 0.
\end{equation}
In general, the expectation of a symmetric random walk is always $0$
since
\begin{displaymath}
  \E(X) = \sum_{x<0} x \Prob\{X=x\} + \sum_{x>0} x \Prob\{X=x\} =
  \sum_{x>0} (x-x) \Prob\{X=x\} = 0.
\end{displaymath}
The \emph{variance} of $X$ is $\Var(X) = \E\left( (X - E(X))^2 \right)
= \E(X^2) - \E(X)^2$. Again, for the example we have
\begin{equation}
  \Var(X) = 1 - 0 = 1.
\end{equation}

Define a new random variable $S$ by summing successive independent
trials of $X$, i.e.
\begin{displaymath}
  S_k = X_1 + X_2 + \cdots + X_k
\end{displaymath}
with the convention $S_0 = 0$. Since the trials are independent, the
expectation of $S_k$ is
\begin{equation}
  \E(S_k) = \E\left(\sum_{i=1}^{k} X_i\right) = \sum_{i=1}^{k} \E(X_i)
  = 0.
\end{equation}
and the variance of $S_k$ is
\begin{equation}
  \Var(S_k) = \Var\left(\sum_{i=1}^{k} X_i\right) = \sum_{i=1}^{k} \Var(X_i).
\end{equation}
For the example, since $\Var(X_i) = 1$, we have $\Var(S_k)=k$.

We say that the $n$th trial $X_n$ occurs at
\emph{epoch}\footnote{Following Feller\cite{FellerI} who follows
  Riordan, the word \emph{epoch} is used to denote \emph{points} on
  the time axis because some contexts use the alternative terms (such
  as moment, time, point) in different meanings.} $n$. We call the
successive partial sums $S_k$ the \emph{positions} of a particle
performing a random walk and mark these values on the vertical axis. A
particular point on the vertical axis will be referred to as a
\emph{site}. For example, Figure~\ref{f:walk} depicts a $40$ step
random walk for $S_k$.

\begin{figure}[ht]
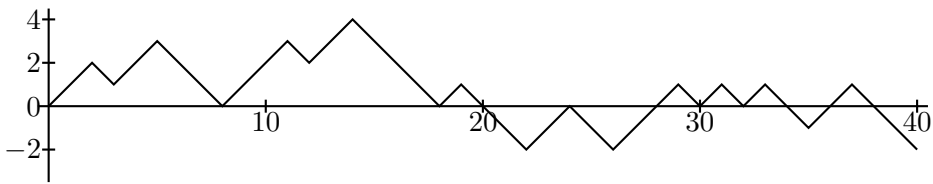

  \begin{center}
    \psset{xunit=.75em}
    \psset{yunit=.75em}
    \pspicture(-1,-4)(41,5)
    \psline(0,-3.5)(0,4.5)
    \psline(-0.3,4)(0.3,4)
    \rput(-0.7,4){$4$}
    \psline(-0.3,2)(0.3,2)
    \rput(-0.7,2){$2$}
    \rput(-0.7,0){$0$}
    \psline(-0.3,-2)(0.3,-2)
    \rput(-1.2,-2){$-2$}
    \psline(-0.5,0)(40.5,0)
    \psline(10,0.3)(10,-0.3)
    \rput(10,-0.7){$10$}
    \psline(20,0.3)(20,-0.3)
    \rput(20,-0.7){$20$}
    \psline(30,0.3)(30,-0.3)
    \rput(30,-0.7){$30$}
    \psline(40,0.3)(40,-0.3)
    \rput(40,-0.7){$40$}
    \psline(0,0)(1,1)(2,2)(3,1)(4,2)(5,3)(6,2)(7,1)(8,0)(9,1)(10,2)(11,3)(12,2)(13,3)(14,4)(15,3)(16,2)(17,1)(18,0)(19,1)(20,0)(21,-1)(22,-2)(23,-1)(24,0)(25,-1)(26,-2)(27,-1)(28,0)(29,1)(30,0)(31,1)(32,0)(33,1)(34,0)(35,-1)(36,0)(37,1)(38,0)(39,-1)(40,-2)
    \endpspicture
    \caption{\label{f:walk} A geometric depiction of the random walk
      $S_k$.}
  \end{center}
\end{figure}

Define a new random variable $\Delta_k$ to be the number of distinct
sites visited during the random walk up to epoch $k$. That is,
\begin{equation}
  \Delta_k = \#\{S_0, \ldots, S_k\} 
  = \max\{S_j\}_{j=0}^{k} - \min\{S_j\}_{j=0}^{k} + 1.
  \label{e:Delta}
\end{equation}
For the example, Theorem~\ref{t:main} becomes the following, first
(foot)noted by P{\'o}lya \cite{Polya1938}.

\begin{theorem}
  For $n$ large, the expectation of $\Delta_n$ is approximately given
  by
  \begin{equation}
    \E(\Delta_n) \sim  \sqrt{\frac{8n}{\pi}} \approx 1.5958 \sqrt{n}.
    \label{e:expect_Delta}
  \end{equation}
  \label{t:Delta_approximation}
\end{theorem}

Therefore an accurate estimate for the expectation of the range of the
walk (the daily high minus low of an instrument) yields an estimate
for the variance (the volatility of the instrument).

The relevance of Theorem~\ref{t:Delta_approximation} to the stock
market hinges upon one's belief that the market behaves like a
symmetric random walk. An alternative but equally reasonable
assumption is that the market behaves like a persistent random
walk. In this case, there is a single parameter $\alpha$ taken between
$0$ and $1$, and the probabilities for the random variable $X$ become
\begin{equation}
  \Prob\{X_i = X_{i-1}\} = \alpha \hspace{3em} 
  \Prob\{X_i = -X_{i-1}\} = 1-\alpha,
  \label{e:prob-persist}
\end{equation}
with the usual convention that $X_1$ satisfies \eqref{e:prob}. For
this case, Theorem~\ref{t:main} becomes the following.

\begin{theorem}
  For $n$ large, the expectation of $\Delta_n$ for a persistent random
  walk with parameter $\alpha$ is approximately given by
  \begin{equation}
    \E(\Delta_n) \sim  \sqrt{\frac{8n\alpha}{\pi(1-\alpha)}} .
    \label{e:persistent_Delta}
  \end{equation}
  \label{t:persistent}
\end{theorem}

Since the variance of a persistent walk is proportional to the
variance of a symmetric walk with constant of proportionality
$\alpha/(1-\alpha)$, Theorem~\ref{t:persistent} follows as a corollary
to Theorem~\ref{t:Delta_approximation}.

\subsection{Expected errors}
\label{s:expect}

It is important to note that \eqref{e:expect_Delta} holds only for the
\emph{expectation} of $\Delta$ and not for any particular instance of
$\Delta$. 

For the example in Figure~\ref{f:walk}, $\Delta_{40} = 7$ whereas the
right hand side of \eqref{e:expect_Delta} is $10.09$. The actual value
of $E(\Delta_{40})$ is approximately $10.16$. Therefore error from
using $\Delta$ in place of $\E(\Delta)$ is due to the variance of
$\Delta$. 

\begin{theorem}
  For $n$ large, the variance of $\Delta_n$ is approximately given
  by
  \begin{equation}
    \Var(\Delta_n) \sim 4n \left(\ln 2 - \frac{2}{\pi}\right) \approx 
    0.2181 n.
    \label{e:var_Delta}
  \end{equation}
  \label{t:var_Delta_approximation}
\end{theorem}

For $n=40$, the right hand side of \eqref{e:var_Delta} is $8.724$, so
this example is more illustrative than exceptional.

When using the range as a surrogate for volatility in a simple moving
average, having a small window can lead to large errors due to the
variance of the range. For instance, if a trader chooses to use a
moving average over a trading month, then he can expect little error
from the variance of the range. However, if he instead chooses to use
only the prior day's historical volatility, then using the range
instead will likely lead to significant errors in the forecast.

\subsection{Approximate errors}
\label{s:approx}

For the sake of concreteness, the following table gives the estimates
and errors for $n \leq 7$.

\newcommand{\rrc}{@{\extracolsep{10pt}}r}
 \begin{table}[ht]
  \begin{center}
    \caption{\label{table} Estimates and errors for $\E(\Delta_n)$.}
    \begin{tabular}{@{\extracolsep{1pt}}c|
        @{\extracolsep{5pt}}r
        \rrc\rrc\rrc\rrc\rrc 
        \rrc\rrc@{\extracolsep{0pt}}}
      \hline
      $n$ & 0 & 1 & 2 & 3 & 4 & 5 & 6 & 7 \\ \hline
      $\E(\Delta_n)$ & 1.0000 & 2.0000 & 2.5000 & 3.0000 & 3.3750 &
      3.7500 & 4.0625 & 4.3750 \\
      $\sqrt{\frac{8n}{\pi}}$ & 0.0000 & 1.5958 & 2.2568 & 2.7640 &
      3.1915 & 3.5682 & 3.9088 & 4.2220 \\
      error & 1.0000 & 0.4040 & 0.2432 & 0.2360 & 0.1835 & 0.1818 &
      0.1537 & 0.1530 \\
      $\%$ error & 100.00 & 20.21 & 9.73 & 7.87 & 5.44 & 4.85 & 3.78 &
      3.50 \\ \hline 
    \end{tabular}
  \end{center}
\end{table}

In order to understand the rate of convergence more precisely, we
present three proofs: using Stirling's formula, using the Tauberian
Theorem, and using properties of the $\Gamma$-function.

\section{An elementary approach}
\label{s:elementary}

In order to prove Theorem~\ref{t:Delta_approximation}, we refine
$\Delta_k$ with a new random variable $\delta_k$ defined by
\begin{equation}
  \delta_k = \left\{ \begin{array}{rl}
      1 & \mbox{a new site is visited at epoch $k$,} \\
      0 & \mbox{otherwise.}
    \end{array} \right.
  \label{e:delta}
\end{equation}
Then clearly we have
\begin{equation}
  \Delta_k = \delta_0 + \delta_1 + \cdots + \delta_k.
  \label{e:deltas}
\end{equation}
If we can derive a formula for $\E(\delta_k)$, then \eqref{e:deltas}
will allow us transform it into a formula for $\E(\Delta_k)$.

To begin, note that since $\delta_k \in \{0,1\}$, we have
\begin{equation}
  \E(\delta_k) = \Prob\{\delta_k = 1\}.
  \label{e:expect_delta}
\end{equation}
Dvoretzky and Erd{\"o}s \cite{DvoretzkyErdos1951} gave the following
alternative interpretation for $\E(\delta_k)$.

\begin{lemma}
  The expectation of $\delta_k$ is given by
  \begin{equation}
    \E(\delta_k) = \Prob\{\mbox{\emph{the origin has not been revisited
        by epoch $n$}}\}.
  \end{equation}
  \label{l:DE}
\end{lemma}

\begin{proof}
  The event $\delta_k=1$ occurs if and only if there is no loop
  beginning at epoch $i$ and returning at epoch $k$ for any
  $i$. Reversing time, this is equivalent to stating that the particle
  does not return to the origin at epoch $k-i$ for any $i$. The result
  now follows from \eqref{e:expect_delta}.
\end{proof}

To make use of this result, we now derive a formula for the
probability that the particle is at the origin at epoch $n$. At this
point we restrict our attention to the running example in order to
make the problem more concrete. As often is the case, it is simpler to
derive a more general formula. For $n$ an epoch and $r$ a site, define
$p_{n,r}$ by
\begin{equation}
  p_{n,r} = \Prob\{\mbox{at epoch $n$, the particle is at site $r$}\}.
\end{equation}
Since a particle can only return to the origin after an even number of
steps, we shall always take $n$ even when $r=0$.

\begin{proposition}
  For $n$ an epoch and $r$ a site, we have
  \begin{equation}
    p_{n,r} = \frac{1}{2^n}\binom{n}{(n+r)/2},
    \label{e:pnr}
  \end{equation}
  where the binomial coefficient is $0$ is the lower term is not an
  integer.
  \label{p:visit}
\end{proposition}

\begin{proof}
  The number of lattice paths with $p$ northeast steps ($\nearrow$)
  and $q$ southeast steps ($\searrow$) is given by
  $\binom{p+q}{p}$. For such a path to go from $(0,0)$ to $(n,r)$, we
  must have $n = p+q$ and $r=p-q$. Therefore the number of lattice
  paths from $(0,0)$ to $(n,r)$ is $\binom{n}{(n+r)/2}$. Dividing by
  $2^n$, the total number of lattice paths with $n$ steps, yields the result.
\end{proof}

In order to apply this result to Lemma~\ref{l:DE}, we must either take
the sum of probabilities $p_{k,0}$ for $1 \leq k \leq n$ or use the
following result.

\begin{lemma}
  The probability that the origin has not been revisited by epoch $2k$
  is equal to the probability that a return to the origin occurs at
  epoch $2k$.
  \label{l:return}
\end{lemma}

\begin{proof}
  We will construct a bijection between paths from $(0,0)$ to $(2k,0)$
  and paths with $2k$ steps that remain weakly above the horizontal
  axis. Given a path from $(0,0)$ to $(2k,0)$, let $(j,m)$ be the
  leftmost occurrence of the minimum site visited. That is, $m \leq
  S_i$ for all $i$ and if equality holds then $j \leq i$. Reflect the
  portion of the path from $(0,0)$ to $(j,m)$ across the vertical line
  $t=j$, and slide the right endpoint of the reflected segment to
  $(2k,0)$. Consider $(j,m)$ to be the resulting path so that it now
  ends at $(2k,2|m|)$. This new path clearly remains weakly above the
  horizontal axis, and the process is easily reversible.

  \begin{figure}[ht]
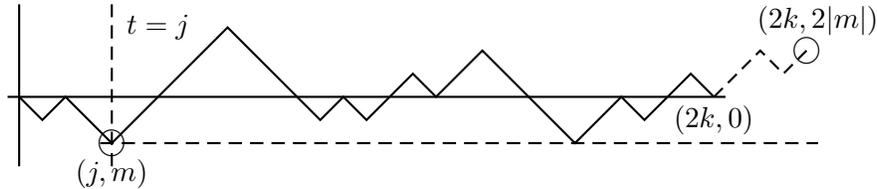

    \begin{center}
      \psset{xunit=.8em}
      \psset{yunit=.8em}
      \pspicture(0,-3)(36,4)
      \psline(0,-3)(0,4)
      \psline(-0.5,0)(30.5,0)
      \rput(30,-1){$(2k,0)$}
      \psline(0,0)(1,-1)(2,0)(4,-2)(9,3)(13,-1)(14,0)(15,-1)%
      (17,1)(18,0)(20,2)(24,-2)(26,0)(27,-1)(29,1)(30,0)
      \psline[linestyle=dashed](4,-3)(4,4)
      \psline[linestyle=dashed](3.5,-2)(34.5,-2)
      \rput(6,3){$t = j$}
      \rput(4,-2){$\bigcirc$}
      \rput(4,-3.3){$(j,m)$}
      \psline[linestyle=dashed](30,0)(32,2)(33,1)(34,2)
      \rput(34,2){$\bigcirc$}
      \rput(34.5,3.3){$(2k,2|m|)$}
      \endpspicture
      \caption{\label{f:nelson} A bijection between paths ending at
        the origin and paths staying weakly above the origin.}
    \end{center}
  \end{figure}
  
  We now claim that the number of paths of length $2k$ that lie
  strictly above the horizontal axis except for the origin is equal to
  one half the number of paths of length $2k$ that lie weakly above
  the horizontal axis. The former paths must all pass through the
  point $(1,1)$ after which they never fall below the horizontal line
  $S=1$. Thus resetting the origin to $(1,1)$ yields a path of length
  $2k-1$ that lies weakly above the horizontal axis. Since $2k-1$ is
  odd, the final point is at least $1$, and so adding another step,
  either northeast ($\nearrow$) or southeast ($\searrow$), the
  resulting path is of length $2k$ and still remains weakly above the
  horizontal axis. Since the last step has probability $1/2$ of either
  option, the claim is proved.

  The lemma now follows from the observation that the number of paths
  of length $2k$ that never return to the origin is twice the number
  of paths of length $2k$ that lie strictly above the horizontal axis
  except for the origin.
\end{proof}

The final ingredient is an approximation for central binomial
coefficients, which we can derive easily from Stirling's formula
\cite{Stirling1730}.

\begin{theorem}[Stirling's formula]
  We have
  \begin{equation}
    n! = \sqrt{2\pi} n^{n+1/2} e^{-n} \left(1 + \frac{1}{12n} +
      \frac{1}{288n^2} + O(n^{-3}) \right) \sim \sqrt{2\pi n}
    \left(\frac{n}{e}\right)^{n}. 
  \end{equation}
\end{theorem}

\begin{theorem}
  For $n$ large, the expectation of $\delta_n$ is approximately given
  by
  \begin{equation}
    \E(\delta_n) \sim  \sqrt{\frac{2}{\pi n}}.
  \end{equation}
  \label{t:delta_approximation}
\end{theorem}

\begin{proof}
  By Lemma~\ref{l:DE}, the expectation of $\delta_n$ is the the
  probability of no return to the origin by epoch $n$. By
  Lemma~\ref{l:return}, for $n$ even this is equal to the probability
  that the particle is at the origin at epoch $n$. Therefore, by
  Proposition~\ref{p:visit}, we have
  \begin{equation}
    \E(\delta_{2k}) = \frac{1}{2^{2k}}\binom{2k}{k}.
    \label{e:p2n}
  \end{equation}
  Letting $n=2k$, Stirling's formula yields the result.
\end{proof}

Finally, to derive Theorem~\ref{t:Delta_approximation} from
Theorem~\ref{t:delta_approximation}, we use \eqref{e:deltas}:
\begin{equation}
  \E(\Delta_n) = \sum_{j=0}^{n} \E(\delta_j) 
  \sim \sqrt{\frac{2}{\pi}} \sum_{j=0}^{n} j^{-1/2} 
  = \sqrt{\frac{2}{\pi}} \left(2 \sqrt{n} \right)
   = \sqrt{\frac{8n}{\pi}}.
\end{equation}

\section{Two generating function approaches}
\label{s:generating_function}

Another approach to Theorem~\ref{t:Delta_approximation} is to derive a
closed form for the generating function of $\E(\Delta_n)$.  Define
the generating functions of $\delta_n$ and $\Delta_n$ by
\begin{equation}
  \delta(z) = \sum_{n \geq 0} \E(\delta_n) z^n
  \hspace{2em} \mbox{and} \hspace{2em}
  \Delta(z) = \sum_{n \geq 0} \E(\Delta_n) z^n.  
\end{equation}
From \eqref{e:deltas}, we have
\begin{equation}
  \Delta(z) = \sum_{n \geq 0} z^n \sum_{j=0}^{n} \E(\delta_j)
  = \sum_{j \geq 0} \E(\delta_j) \sum_{n \geq j} z^n
  = \frac{\delta(z)}{1-z} . 
  \label{e:Delta_z}
\end{equation}
Therefore it suffices to find a closed formula for $\delta(z)$.

Define $f_{n,r}$ to be the probability that the particle has first
reached site $r$ at epoch $n$. Then we have
\begin{equation}
  \E(\delta_k) = \Prob\{\delta_k=1\} = \sum_{r} f_{k,r}.
\end{equation}
Letting $F_r(z) = \sum_{n \geq 0} f_{n,r} z^n$ be the generating
function for $f_{n,r}$, we have
\begin{equation}
  \delta(z) = \sum_{r} F_r(z).
  \label{e:delta_F}
\end{equation}

Recall $p_{n,r}$ denotes that probability that the particle is at site
$r$ at epoch $n$. This probability may be decomposed into the first
arrival at site $r$, say at epoch $j$, followed by a loop back $n-j$
steps later. That is,
\begin{equation}
  p_{n,r} = \sum_{j=0}^{n} f_{j,r} p_{n-j,0}.
\end{equation}
Letting $P_r = \sum_{n \geq 0} p_{n,r} z^n $ be the generating
functions for $p_{n,r}$, this becomes
\begin{equation}
  P_r(z) = F_r(z) P_0(z).
  \label{e:PF}
\end{equation}

Returning now to $\delta(z)$, solving \eqref{e:PF} for $F_r(z)$ and
combining the result with \eqref{e:delta_F} yields
\begin{displaymath}
  \delta(z) = \sum_{r} \frac{P_r(z)}{P_0(z)}
  = \frac{1}{P_0(z)} \sum_{r} \sum_{n \geq 0} p_{n,r} z^n 
  = \frac{1}{P_0(z)} \sum_{n \geq 0} z^n \left(\sum_{r} p_{n,r}\right).
\end{displaymath}
The inner sum on the right hand side is the probability that some site is
visited at epoch $n$, which is a certainty. Therefore
\begin{equation}
  \delta(z) = \frac{1}{P_0(z)} \sum_{n \geq 0} z^n = \frac{1}{P_0(z)}\frac{1}{1-z}.
\end{equation}
Substituting back into \eqref{e:Delta_z}, we have proved the following.

\begin{theorem}
  The generating function for $\E(\Delta_n)$ is given by
  \begin{equation}
    \Delta(z) = \sum_{n \geq 0} \E(\Delta_n) z^n = \frac{1}{P_0(z)}\frac{1}{(1-z)^2}.
    \label{e:Delta_Z}
  \end{equation}
  \label{t:Delta_Z}
\end{theorem}

By Proposition~\ref{p:visit}, we deduce a closed form for $P_0(z)$ for
the simple walk,
\begin{displaymath}
  P_0(z) = \sum_{n \geq 0} p_{2n,0}z^{2n} = \sum_{n \geq 0}
  \frac{1}{2^{2n}} \binom{2n}{n} z^{2n} = \sum_{n \geq 0}
  \binom{1/2}{n} (z^{2})^{n} = \frac{1}{\sqrt{1-z^2}}.
\end{displaymath}
Here we have used the basic identity $(1-\zeta)^m = \sum_{k \geq 0}
\binom{-m}{k} \zeta^k$. Substituting this into \eqref{e:Delta_z} gives
\begin{equation}
  \Delta(z) = \frac{\sqrt{1-z^2}}{(1-z)^2}.
  \label{e:DZ}
\end{equation}

With a closed form for the generating function, we can get an estimate
on the coefficients. To do this, it is helpful to rewrite
\eqref{e:Delta_Z} as
\begin{equation}
  \Delta(z) = \frac{\sqrt{2}}{(1-z)^{3/2}} \left(1 + \frac{1}{2}(1-z)\right)^{1/2}.
  \label{e:Delta_estimate}
\end{equation}

\subsection{Using the Tauberian Theorem}
\label{s:tauberian}

Our first estimate uses the powerful Tauberian Theorem
\cite{Tauber1897}.

\begin{theorem}[Tauberian Theorem]
  Let $\{q_n\}$ be a monotone, nonnegative sequence with generating
  function $Q(z) = \sum_{n \geq 0} q_n z^n$. Then for $\rho>0$ and $L$
  a slowly varying function, we have
  \begin{displaymath}
    \displaystyle{Q(z) \sim \frac{1}{(1-z)^{\rho}} L(\frac{1}{1-z})}
    \mbox{ as } \displaystyle{z \rightarrow 1^{-}}
    \hspace{1em} \mbox{iff} \hspace{1em}
    \displaystyle{q_n \sim \frac{1}{\Gamma(\rho)} n^{\rho-1} L(n)}
    \mbox{ as } \displaystyle{n \rightarrow \infty}
  \end{displaymath}
  Here \emph{slowly varying} means that for all $\lambda$, $L(\lambda
  x)/L(x) \rightarrow 1$ as $x \rightarrow \infty$.
\end{theorem}

We may now derive Theorem~\ref{t:Delta_approximation} from
\eqref{e:Delta_estimate} and the Tauberian Theorem with $\rho = 3/2$
and $L(x) = \sqrt{2}$:
\begin{equation}
  \E(\Delta_n) \sim \frac{1}{\Gamma(3/2)} n^{1/2} \sqrt{2}
  = \sqrt{\frac{8n}{\pi}}.
\end{equation}
Unfortunately, with such a general and powerful theorem, there are not
good estimates on the error of the approximation. 

\subsection{Using the $\Gamma$-function}
\label{s:Gamma}

Since the closed form of $\Delta(z)$ is fairly simple, we can use
properties of the $\Gamma$ function directly to find the error
terms. We begin by expanding \eqref{e:Delta_estimate} as follows:
\begin{eqnarray*}
  \Delta(z) & = & \frac{\sqrt{2}}{(1-z)^{3/2}} \left(1 +
    \frac{1}{2}(1-z)\right)^{1/2} \\
  & = & \frac{\sqrt{2}}{(1-z)^{3/2}} \sum_{k \geq 0}
    \binom{-1/2}{k}\left(\frac{1}{2}(1-z)^k\right) \\
  & = & \sqrt{2} \sum_{k \geq 0} \binom{-1/2}{k}
  \left(\frac{1}{2}\right)^k (1-z)^{k-3/2} \\
  & = & \sqrt{2} \sum_{k \geq 0} \binom{-1/2}{k}
  \left(\frac{1}{2}\right)^k \sum_{n \geq 0} \binom{3/2-k}{n} z^n.
\end{eqnarray*}
Isolating the coefficient of $z^n$ and manipulating using the $\Gamma$
function gives
\begin{eqnarray*}
  \E(\Delta_n) & = & \sqrt{2} \sum_{k \geq 0} \binom{-1/2}{k}
  \binom{3/2-k}{n} \left(\frac{1}{2}\right)^{k} \\
  & = & \sqrt{2} \sum_{k \geq 0} \binom{-1/2}{k}
  \left(\frac{1}{2}\right)^{k} \frac{\Gamma(3/2-k+n)}{\Gamma(3/2-k)n!} \\
  & = & \frac{\sqrt{2}}{\Gamma(3/2)}\frac{\Gamma(3/2+n)}{n!} \sum_{k \geq 0}
  \binom{-1/2}{k} \left(\frac{1}{2}\right)^{k}
  \frac{\Gamma(3/2-k+n)\Gamma(3/2)}{\Gamma(3/2-k)\Gamma(3/2+n)}
\end{eqnarray*}

We can manipulate the summand using the fundamental property of
$\Gamma$
\begin{displaymath}
  \Gamma(z+1) = z\Gamma(z),
\end{displaymath}
and we can approximate the outer term using the estimate
\begin{displaymath}
  \frac{\Gamma(n+a)}{n!} = n^{a-1} \left(1 + \frac{a(a-1)}{2n} +
    \frac{a(a-1)(a-2)(3a-1)}{24n^2} + O(n^{-3})\right).
\end{displaymath}
Combining all of these gives
\begin{equation}
  \E(\Delta_n) = \sqrt{\frac{8n}{\pi}} \left(1 +
    \frac{1}{4n} - \frac{1}{32n^2} + O(n^{-3}) \right).
\end{equation}

\vspace{\baselineskip}

\begin{center}
  \textsc{Acknowledgments}
\end{center}

The author is grateful to Elwyn Berlekamp for introducing her to this
problem and to the wonders of the world of finance, and to Sean Borman
for helpful feedback on early drafts.

\bibliographystyle{abbrv} 
\bibliography{rangeOfVariance}

\end{document}